\documentclass[11pt]{article}

\usepackage{amsthm}
\usepackage{amsmath}
\usepackage{amssymb}
\usepackage[margin=1in]{geometry}
\usepackage[backref=page]{hyperref}
\usepackage{graphicx}
\usepackage{booktabs}
\usepackage{enumitem}

\hypersetup{
    colorlinks=true,     
    linkcolor=blue,      
    citecolor=blue,      
    filecolor=blue,      
    urlcolor=blue        
}

\newtheorem{theorem}{Theorem}
\newtheorem{lemma}{Lemma}
\newtheorem{definition}{Definition}

\newtheorem{remark}{Remark}

\renewcommand{\S}{\mathbf{S}}

\newcommand{\psd}{\succeq}

\newcommand{\RR}{\mathbb{R}}

\newcommand{\EE}{\mathbb{E}}

\DeclareMathOperator{\Tr}{Tr}
\DeclareMathOperator{\conv}{conv}

\DeclareMathOperator{\diag}{diag}

\DeclareMathOperator{\proj}{proj}
\DeclareMathOperator*{\E}{\mathbf{E}}
\let\Pr\relax
\DeclareMathOperator*{\Pr}{\mathbf{Pr}}
\DeclareMathOperator{\xc}{\mathbf{xc}}

{\begin{small}\begin{list}{}%
         {\setlength{\leftmargin}{1cm}\setlength{\rightmargin}{1cm}}%
         \item[]%
}
{\end{list}\end{small}}

\title{On polyhedral approximations of the positive semidefinite cone}
\author{Hamza Fawzi\thanks{Department of Applied Mathematics and Theoretical Physics, University of Cambridge. Email: \texttt{h.fawzi@damtp.cam.ac.uk}.}
}
\date{November 23, 2018}

\newcommand{\Den}{D}

\begin{document}

\maketitle

\begin{abstract}
Let $\Den$ be the set of $n\times n$ positive semidefinite matrices of trace equal to one, also known as the set of density matrices. We prove two results on the hardness of approximating $\Den$ with polytopes. First, we show that if $0 < \epsilon < 1$ and $A$ is an arbitrary matrix of trace equal to one, any polytope $P$ such that $(1-\epsilon)(\Den-A) \subset P \subset \Den-A$ must have linear programming \emph{extension complexity} at least $\exp(c\sqrt{n})$ where $c > 0$ is a constant that depends on $\epsilon$. Second, we show that any polytope $P$ such that $\Den \subset P$ and such that the \emph{Gaussian width} of $P$ is at most twice the Gaussian width of $\Den$ must have extension complexity at least $\exp(cn^{1/3})$. The main ingredient of our proofs is hypercontractivity of the noise operator on the hypercube.
\end{abstract}


\section{Introduction}

Let $\Den$ be the set of positive semidefinite matrices of trace equal to one, also known as the set of density of matrices:
\[
\Den = \left\{ X \in \S^n_+, \Tr(X) = 1 \right\} = \conv \left \{ xx^T : x \in \RR^n, \|x\|_2 = 1 \right\}
\]
where $\S^n_+$ is the set of $n\times n$ real symmetric positive semidefinite matrices. It is well-known that $\Den$ is not polyhedral.
We consider the following question: how well can we approximate $\Den$ with a polytope of small complexity?

To make this question precise, we need to specify the notion of \emph{approximation}, as well as the measure of \emph{complexity} of polytopes. 
For the latter we use the \emph{extension complexity} of polytopes defined as follows:
\begin{definition}
\label{def:xc}
The \emph{extension complexity} of a polytope $P$ is the smallest integer $k$ such that $P$ can be expressed as the linear image of a polytope with $k$ facets.
\end{definition}
The extension complexity was first studied systematically by Yannakakis in \cite{yannakakis1991expressing} and has recently attracted a lot of attention, see e.g., \cite{barvinok2006computational,kaibel2011extended,pashkovich2012extended,litvak2014approximation,fiorini2015exponential,rothvoss2014matching,fawziphdthesis,kothari2017approximating}. This definition is motivated by computational aspects: if $P$ has extension complexity $N$, then optimizing a linear function on $P$ can be formulated as a \emph{linear program} of size $N$. The extension complexity is always smaller than the number of vertices and facets, and in fact it can be exponentially smaller (see comments below for more on this). The extension complexity is also invariant under polarity.

For the definition of \emph{approximation}, we consider in this paper two different notions. First, we consider polytopes $P$ that satisfy $(1-\epsilon)  (\Den-A) \subset P \subset \Den-A$ where $A$ is an arbitrary matrix of trace equal to one. We now state our first result:
\begin{theorem}
\label{thm:main}
For any $0 < \epsilon < 1$, there exists a constant $c > 0$ such that the following is true. Let $A$ be an $n\times n$ real symmetric matrix of trace equal to one. If $P$ is a polytope such that $(1-\epsilon) (\Den-A) \subset P \subset \Den-A$ then the extension complexity of $P$ is at least $e^{c \sqrt{n}}$.
\end{theorem}

Our second theorem deals with a different notion of approximation. Recall that the \emph{Gaussian width} of a set $S$ in Euclidean space is defined as 
\[ w_G(S) := \EE_{g}\left[\max_{x \in S} \langle g, x \rangle\right] \]
 where $g$ is a standard Gaussian random vector. In the theorem below, we consider polytopes $P$ that contain $\Den$ and such that the Gaussian width of $P$ is not much larger than that of $\Den$.
   We prove:

\begin{theorem}
\label{thm:gw}
If $P$ is a polytope such that $\Den \subset P$ and $w_G(P) \leq 2 w_G(\Den)$ then the extension complexity of $P$ is at least $e^{c n^{1/3}}$, where $c > 0$ is an absolute constant.
\end{theorem}
We now comment on the statement of these theorems and related work.
\begin{itemize}
\item We do not know if the lower bounds in Theorems \ref{thm:main} and \ref{thm:gw} are tight. For the setting of Theorem \ref{thm:main}, and with the origin at $A = \frac{1}{n} I_n$, the best construction we know of has size $\exp(cn)$ and corresponds to $P$ being the convex hull of $\{x_i x_i^T\}$ where $\{x_i\}$ is a well-chosen $\delta$-net of the unit sphere $S^{n-1}$ with $\delta > 0$ being a constant that depends on $\epsilon$, see \cite[Proposition 10]{aubrun2017dvoretzky}.\footnote{We note that the $\delta$-net has to be chosen suitably for this construction to work and we refer to \cite[Proposition 10]{aubrun2017dvoretzky} where it is proved that a random net works. The trivial construction consists in taking a $\delta$-net where $\delta \approx \frac{1}{n}$ which gives a polytope with $\exp(c n\log n)$ vertices} If we measure the complexity of a polytope by the \emph{number of vertices}, then this is the best possible construction; in other words any polytope contained between $(1-\epsilon)(\Den - \frac{1}{n} I_n)$ and $\Den - \frac{1}{n} I_n$ for constant $\epsilon$ must have an exponential number of vertices.
\item Let $B^n$ be the unit Euclidean ball in $\RR^n$. It is known (e.g., by lower bounds on $\delta$-nets of the sphere) that any polytope $P$ that satisfies $(1-\epsilon) B^n \subset P \subset B^n$ must have an exponential number of vertices, for constant $\epsilon > 0$. In constrast, Ben-Tal and Nemirovski \cite{ben2001polyhedral} showed that there is a polytope $P$ of extension complexity \emph{linear in $n$} (more precisely $O(n \log(1/\epsilon))$) that satisfies $(1-\epsilon) B^n \subset P \subset B^n$. Unlike the ball $B^n$, Theorem \ref{thm:main} shows that $\Den$ is hard to approximate even in terms of extension complexity. We note in passing that for $n=2$ the set of density matrices coincides precisely with $B^2$, the disk.
\item In \cite{braun2012approximation} it was shown that there is a convex set $S \subset \RR^{n\times n}$ that has \emph{positive semidefinite rank}\footnote{A convex set $S$ has positive semidefinite rank $k$ if it can be written as the linear image of an affine section of the $k\times k$ positive semidefinite cone. See e.g., \cite{psdranksurvey}.} $n+1$ and such that any polytope $P$ that satisfies $S \subset P \subset S^{\epsilon}$ must have extension complexity at least $\exp(cn)$ for some constant $c > 0$, where $S^{\epsilon} = \{y : \exists x \in S, \|y-x\|_1 \leq \epsilon\}$ is the $\epsilon$-widening of $S$ in the $\ell_1$ norm and $\epsilon > 0$ is a constant. We omit the exact definition of $S$ here and we refer the reader to \cite{braun2012approximation} for the details; we can just mention that $S$ satisfies $\Tr(X) \leq n$ for all $X \in S$. The lower bound on $S$ does not, as far as we know, directly imply a lower bound for approximating $\Den$.

\item In quantum information theory the set of density matrices is defined as the set of \emph{Hermitian} positive semidefinite matrices of trace equal to one. We have considered here real symmetric matrices for convenience, but the results easily apply to the Hermitian case too.

\item Theorems \ref{thm:main} and \ref{thm:gw} are also true (with the same proof) if we replace $\Den$ with the (scaled) \emph{elliptope} $E = \{X \in \S^n : X \psd 0 \text{ and } X_{ii}=1/n \; \forall i=1,\ldots,n\} \subset \Den$. This is because the matrix we consider in Equation \eqref{eq:Slackcube} is also a submatrix of the slack matrix of the elliptope, and because the Gaussian width of $E$ is, up to a constant, the same as that of $\Den$.
\end{itemize}

\paragraph{Organization} Our main tool to prove Theorems \ref{thm:main} and \ref{thm:gw} is hypercontractivity for the noise operator on the hypercube. Hypercontractivity has been a major tool in theoretical computer science \cite{odonnellbook} and was used in particular to derive lower bound on communication complexity. Our proof draws inspiration from some of these works, particularly \cite{regev2011quantum}. In Section \ref{sec:bg} we present some background material on hypercontractivity as well as nonnegative factorizations and slack matrices. Theorem \ref{thm:main} is proved in Section \ref{sec:proofmain} and Theorem \ref{thm:gw} is proved in Section \ref{sec:proofgw}.

\section{Background}
\label{sec:bg}

We review necessary background material on hypercontractivity, and the connection between extension complexity and the nonnegative rank of slack matrices.

\subsection{Hypercontractivity}

Let $H_n = \{-1,1\}^n$ be the discrete hypercube in $n$ dimensions. Any function $f:H_n\rightarrow \RR$ has a Fourier expansion
\[
f = f_0 + f_1 + \dots + f_n
\]
where each $f_i$ is a homogeneous multilinear polynomial of degree $i$.

\paragraph{Noise operator} Given $\rho \in [0,1]$ the noise (or smoothing) operator $T_{\rho}$ acts on a function $f:H_n\rightarrow \RR$ by multiplying the $i$'th Fourier term by $\rho^i$, namely:
\[
T_{\rho} f = f_0 + \rho f_1 + \rho^2 f_2 + \dots + \rho^n f_n.
\]
By attenuating the high-frequency terms, the function $T_{\rho} f$ is a ``smoother'' version of $f$. The parameter $\rho \in [0,1]$ controls the level of the smoothing operation: for $\rho=0$ the resulting function $T_{\rho} f$ is constant equal to $\E[f]$, and for $\rho=1$ no smoothing happens.

The hypercontractive inequality gives us a quantitative estimate of how smoother $T_{\rho} f$ is compared to $f$. Before stating the result, let $\mu$ be the uniform probability measure on $H_n$, and for $f : H_n\rightarrow \RR$ and $p \geq 1$, define
\[
\|f\|_p = \left(\E_{x \sim \mu} [ |f(x)|^p ]\right)^{1/p}.
\]
Note that $\|f\|_p \leq \|f\|_q$ for $p \leq q$. The ratios $\|f\|_q / \|f\|_p$ for $q > p$ quantify the smoothness (or flatness) of a function $f$. The hypercontractive inequality for $T_{\rho}$, due to Bonami and Beckner, can then be stated as follows \cite{bonami1970etude,beckner1975inequalities}.
\begin{theorem}[Hypercontractivity of the noise operator on the hypercube] For any $0 < \rho \leq 1$, $f:H_n \rightarrow \RR$, and $p \geq 1$ we have $\|T_{\rho} f\|_q \leq \|f\|_p$ where $q = 1 + \rho^{-2}(p-1) \geq p$.
\end{theorem}

Hypercontractivity has been a major tool in theoretical computer science and we refer the reader to \cite{odonnellbook} for more details.

\subsection{Slack matrices, nonnegative rank}

In this section we briefly review some background material concerning the extension complexity (Definition \ref{def:xc}) and its connection to the nonnegative rank of matrices. For more details, we refer to e.g., \cite{gouveia2011lifts}.

Consider two convex bodies $K \subset L \subset \RR^d$ with the origin in the interior of $K$. Let $L^{\circ}$ be the polar of $L$ so that $L$ has the following inequality description:
\begin{equation}
\label{eq:Lpolar}
L = \left\{x \in \RR^d : \langle y , x \rangle \leq 1 \; \forall y \in L^{\circ} \right\}.
\end{equation}
Define the (potentially infinite) \emph{slack matrix} of the pair $(K,L)$ as the matrix
\begin{equation}
\label{eq:slackKL}
S_{K,L}(x,y) \mapsto 1 - \langle x , y \rangle
\end{equation}
where rows are indexed by extreme points $x$ of $K$, and columns are indexed by extreme points $y$ of $L^{\circ}$. Since $K \subset L$, $S_{K,L}(x,y) \geq 0$ for all $x,y$. The definition of \emph{nonnegative rank} for elementwise nonnegative matrices will be crucial for the rest of the paper:
\begin{definition}[Nonnegative rank]
Let $M$ be an elementwise nonnegative matrix. The \emph{nonnegative rank} of $M$ is the smallest $k$ such that $M$ can be written as a sum of $k$ rank-one matrices that are elementwise nonnegative.
\end{definition}

The following theorem relates the extension complexity of any polytope $P$ contained between $K$ and $L$ to the nonnegative rank of $S_{K,L}$. It is due to Yannakakis for the special case $K=L$ \cite{yannakakis1991expressing}; the proof of the more general case is essentially the same and can be found  e.g., in \cite[Section 4.1]{pashkovich2012extended}.

\begin{theorem}[Generalized Yannakakis theorem]
Let $K \subset L$ be two convex bodies and assume that the origin lies in the interior of $K$. Let $S_{K,L}$ be the slack matrix of the pair $(K,L)$ as defined in Equation \eqref{eq:slackKL}. The smallest extension complexity of a polytope $P$ that satisfies $K \subset P \subset L$ is equal to the nonnegative rank of $S_{K,L}$.
\end{theorem}

\section{Proof of Theorem \ref{thm:main}}
\label{sec:proofmain}

\paragraph{Reduction to the case $A = \frac{1}{n} I_n$} We first show that it suffices to consider the case $A = \frac{1}{n} I_n$. The following argument is due to Guillaume Aubrun and we are including it with his permission. Assume $P$ is a polytope such that $(1-\epsilon) (\Den - A) \subset P \subset \Den - A$ where $A$ is a symmetric matrix of trace equal to one. By applying a rotation we can assume without loss of generality that $A$ is diagonal. Let $S$ be the permutation matrix for the cycle $(1,2,\ldots,n)$. Since $A$ is diagonal we have $\frac{1}{n} I_n = \frac{1}{n} \sum_{i=1}^n S^i A S^{-i}$. Since, furthermore $S^i \Den S^{-i} = \Den$, we get $(1-\epsilon) (\Den - \frac{1}{n} I_n) \subset Q \subset \Den - \frac{1}{n} I_n$ where $Q = \frac{1}{n} \sum_{i=1}^n S^i P S^{-i}$ (Minkowski sum). It is known and easy to verify that the extension complexity of the Minkowski sum of two polytopes is at most the sum of the extension complexities. If we denote the extension complexity by $\xc$ this yields $\xc(Q) \leq n \xc(P)$. Thus if $\xc(Q) \geq e^{c\sqrt{n}}$ we get $\xc(P) \geq \exp(c\sqrt{n})/n \geq \exp(c'\sqrt{n})$.

\begin{remark}[Elliptope]
If we consider the elliptope $E$ instead of $\Den$ then the argument above needs to be slightly modified since applying a rotation matrix to make $A$ diagonal will change $E$. We thus give an alternative argument for this case. Assume $A$ is a symmetric matrix of trace equal to one, and assume that $(1-\epsilon) (E - A) \subset P \subset E-A$. Since the identity matrix lies in the convex hull of $\{xx^T : x \in \{-1,1\}^n\}$ there exist numbers $\lambda_1,\ldots,\lambda_K \geq 0$ and vectors $x_1,\ldots,x_K \in \{-1,1\}^n$ such that such that $\sum_{i=1}^{K} \lambda_i =1$ and $\sum_{i=1}^K \lambda_i x_i x_i^T = I_n$, with $K \leq n^2$ (by Carath\'eodory theorem). Note that $\sum_{i=1}^K \lambda_i \diag(x_i) A \diag(x_i) = \diag(A)$ and that $\sum_{i=1}^{K} \lambda_i \diag(x_i) E \diag(x_i) = E$. Thus if we let $P' = \sum_{i=1}^{K} \lambda_i \diag(x_i) P \diag(x_i)$ then $(1-\epsilon) (E-\diag(A)) \subset P' \subset E-\diag(A)$ and $\xc(P') \leq n^2 \xc(P)$. Using the same argument as above with the cyclic permutation matrix we get that $(1-\epsilon) (E-\frac{1}{n} I_n) \subset Q \subset E-\frac{1}{n} I_n$ with $\xc(Q) \leq n\xc(P') \leq n^3 \xc(P)$. Thus if $\xc(Q) \geq e^{c\sqrt{n}}$ we get that $\xc(P) \geq e^{c'\sqrt{n}}$.
\end{remark}

In the rest of this section we thus focus on the case $A = \frac{1}{n} I_n$.

\paragraph{Slack matrix}  We start by computing the slack matrix of the pair $((1-\epsilon) C, C)$ where $C = \Den - \frac{1}{n} I_n$. The extreme points of $C$ are the $xx^T - \frac{1}{n} I_n$ where $x \in S^{n-1}$. One can verify that the polar of $C$ (computed in the space of trace-zero matrices) is $C^{\circ} = -nC$, and its extreme points are $I_n - nyy^T$ where $y \in S^{n-1}$. By checking that $1 - \left\langle (1-\epsilon)(xx^T - \frac{1}{n} I_n) , I_n - nyy^T \right\rangle = (1-\epsilon) n (x^T y)^2 + \epsilon$ we are thus led to study the nonnegative rank of the following infinite matrix:
\begin{equation}
\label{eq:Slacksphere}
(x,y) \mapsto (1-\epsilon) n (x^T y)^2 + \epsilon,
\quad (x \in S^{n-1}, y \in S^{n-1})
\end{equation}
where rows and columns are indexed by elements of the $(n-1)$-sphere. It will be enough for us to consider the submatrix indexed by elements of the hypercube $\{-\frac{1}{\sqrt{n}}, +\frac{1}{\sqrt{n}}\}^n$. In other words we will work with the $2^n \times 2^n$ matrix:\footnote{We can also work directly with the full matrix \eqref{eq:Slacksphere} defined on $S^{n-1} \times S^{n-1}$ and use hypercontractivity of the Poisson kernel on the sphere \cite{beckner1992sobolev} instead of the noise operator on the hypercube. The proof is identical and gives the same bounds. We decided to work on the hypercube to avoid technical issues having to do with convergence of Fourier series for functions on $S^{n-1}$.}
\begin{equation}
\label{eq:Slackcube}
(x,y)\mapsto (1-\epsilon) \frac{1}{n} (x^T y)^2 + \epsilon \quad (x \in \{-1,1\}^n, y \in \{-1,1\}^n).
\end{equation}

\paragraph{Nonnegative factorization} We now proceed to prove a lower bound on the nonnegative rank of \eqref{eq:Slackcube}. Assume we can write the matrix \eqref{eq:Slackcube} as a sum of $N$ rank-one nonnegative terms:
\begin{equation}
\label{eq:NMF1}
(1-\epsilon) \frac{1}{n} (x^T y)^2 + \epsilon = \sum_{i=1}^N f_i(x) g_i(y) \quad \forall x,y \in H_n,
\end{equation}
where $f_i,g_i:H_n\rightarrow \RR_+$ are nonnegative functions on $H_n$. Since $\E_{x \in H_n} [ (1-\epsilon) \frac{1}{n} (x^T y)^2 + \epsilon ] = 1$ for all fixed $y \in H_n$, we can scale the $f_i$ so that $\E[f_i] = 1$ for all $i=1,\ldots,N$, and $\sum_{i=1}^N g_i \equiv 1$.

For each fixed $y$, let $q_y(x) = \frac{1}{n} (x^T y)^2$. We can apply the noise operator $T_{\rho}$ on both sides of \eqref{eq:NMF1}, seen as functions of $x \in H_n$ (for $y$ fixed), to get:
\begin{equation}
\label{eq:NMF1-Trho}
(1-\epsilon) T_{\rho} q_y(x) + \epsilon = \sum_{i=1}^N T_{\rho}f_i(x) g_i(y)
\quad
\forall x, y \in H_n.
\end{equation}
It is easy to verify that for fixed $y$, $T_{\rho} q_y(x) = 1 + \rho^2 (q_y(x)-1)$. Using this fact, and plugging $x = y$ in \eqref{eq:NMF1-Trho} we get that, using $q_y(y) = n$:
\[
(1-\epsilon) ( 1 + \rho^2 ( n-1) ) + \epsilon = \sum_{i=1}^N T_{\rho} f_i(y) g_i(y).
\]
For concreteness we will assume in the rest of the proof that $\epsilon = 1/3$ (it is easy to see that the proof works with any constant $0 < \epsilon < 1$). Then if we take $\rho = \sqrt{5/n}$ we get:
\begin{equation}
\label{eq:NMF1-Trho-eval}
4 \leq \frac{2}{3} \left( 1 + \frac{5}{n} (n-1) \right) + \frac{1}{3} = \sum_{i=1}^N T_{\rho} f_i(y) g_i(y) \quad \forall y \in H_n.
\end{equation}
Since $\sum_{i=1}^N g_i(y) = 1$, Equation \eqref{eq:NMF1-Trho-eval} implies that for all $y \in H_n$ there is at least one $i \in \{1,\ldots,N\}$ such that $(T_{\rho} f_i)(y) \geq 4$. We now invoke hypercontractivity to show that the functions $T_{\rho} f_i$ must be highly concentrated around their mean $\E[T_{\rho} f_i] = \E[f_i] = 1$, provided $f_i$ is not too ``spiked'' to start with. This is the content of the next lemma, which is key to our proof:

\begin{lemma}
\label{lem:concentration1}
Assume $f : H_n \rightarrow \RR$ with $f \geq 0$ such that $\E[f] = 1$ and $\max f \leq e^{\sqrt{n}}$. Assume $\rho = \sqrt{5/n}$. Then
\[
\mu \{ x \in H_n : (T_{\rho} f)(x) \geq 4 \} \leq \exp(-c\sqrt{n})
\]
for some absolute constant $c > 0$.
\end{lemma}
\begin{proof}
For any choice of $p \geq 1$ and $q = 1+\rho^{-2}(p-1)$ we have, by Markov's inequality:
\[
\mu \{ x \in H_n : (T_{\rho} f)(x) \geq t \} = \mu \{ x \in H_n : (T_{\rho} f)(x)^{q} \geq t^{q} \} \leq \frac{\E[(T_{\rho} f)^q]}{t^{q}} \leq \frac{\|f\|_{p}^{q}}{t^{q}}.
\]
Let $M = \max f$. Note that
\[
\|f\|_p = (\E [ f(x)^p ])^{1/p} \leq (M^{p-1} \E[f])^{1/p} \leq M^{(p-1)/p} \leq M^{p-1}.
\]
Putting this in the previous inequality gives us
\[
\mu \{ x : (T_{\rho} f)(x) \geq t \} \leq \frac{M^{q(p-1)}}{t^{q}}.
\]
We take $p = 1 + \frac{1}{\log M}$ so that $q = 1 + \frac{1}{5}\frac{n}{\log M}$. With this choice we have $M^{p-1} = \exp(1)$ so that
\[
\mu \{ x : (T_{\rho} f)(x) \geq t \} \leq \left(\frac{e}{t}\right)^{1 + \frac{n}{5 \log M}}.
\]
Using the fact that $\log M = \sqrt{n}$ and $e/t = e/4 < 1$ we get the desired result.
\end{proof}

If all our functions $f_i$ satisfy $\max f_i \leq e^{\sqrt{n}}$ then the previous lemma concludes the argument. Indeed, from \eqref{eq:NMF1-Trho-eval} we know that for each $y \in H_n$ there is at least $i$ such that $T_{\rho} f_i(y) \geq 4$. But since for all $i$, $\mu\{y : T_{\rho} f_i \geq 4\} \leq \exp(-c\sqrt{n})$ it must be that we have at least $\exp(c\sqrt{n})$ functions $f_i$.

The rest of the proof is to deal with the case where some of the functions have a large maximum, larger than $\exp(\sqrt{n})$. In this case the idea is to go back to the identity \eqref{eq:NMF1} and to note that, since $f_i,g_i \geq 0$ it must be that $\|f_i \|_{\infty} \|g_i\|_{\infty} \leq n$ for all $i=1,\ldots,N$. It thus follows that if $\|f_i\|_{\infty}$ is large, then $\|g_i\|_{\infty}$ is small and as such the weight of the $i$'th function in \eqref{eq:NMF1-Trho-eval} is ``negligible''. We make this precise now.

Let $I = \{i \in \{1,\ldots,N\} : \|f_i\|_{\infty} \geq e^{\sqrt{n}}\}$. We rewrite \eqref{eq:NMF1-Trho-eval} as follows:
\begin{equation}
\label{eq:sepmf}
\forall y \in H_n, \quad 
4 \leq \sum_{i \in I} g_i(y) (T_{\rho} f_i)(y) + \sum_{i \in I^c} g_i(y) (T_{\rho} f_i)(y).
\end{equation}
We can assume that $|I| \leq e^{\sqrt{n}/4}$, because otherwise there is nothing to show. We will show that the contribution of the first term in \eqref{eq:sepmf} is negligible for most values of $y$, because $\|g_i\|_{\infty} \leq n e^{-\sqrt{n}}$ for $i \in I$.
Let $A = \{y \in H_n : (T_{\rho} f_i)(y) \leq e^{\sqrt{n}/2} \; \forall i \in I\}$. For any $y \in A$ we have
\[
\sum_{i \in I} g_i(y) (T_{\rho} f_i)(y) \leq ne^{-\sqrt{n}} e^{\sqrt{n}/2} e^{\sqrt{n}/4} \leq 1
\]
for large enough $n$. So for all $y \in A$ we have
\[
\sum_{i \in I^c} g_i(y) (T_{\rho} f_i)(y) \geq 4-1 = 3.
\]
Using the concentration lemma (Lemma \ref{lem:concentration1}, where we use $t=3 > e$ in the proof instead of $t=4$) we know that for $i \in I^c$, $\mu\{y : T_{\rho} f_i(y) \geq 3\} \leq \exp(-c\sqrt{n})$. Thus this tells us that we must have:
\begin{equation}
\label{eq:Iclb}
|I^c| \geq \mu(A) \exp(c\sqrt{n}).
\end{equation}
It remains to evaluate $\mu(A)$. This can be done by a simple application of Markov's inequality. We know that for any $i$, $\mu\{T_\rho f_i \geq e^{\sqrt{n}/2}\} \leq e^{-\sqrt{n}/2}$ thus by the union bound we get $\mu(A) \geq 1-|I|e^{-\sqrt{n}/2}$. Combining with \eqref{eq:Iclb}
\[
|I^c| \geq (1-|I|e^{-\sqrt{n}/2})\exp(c\sqrt{n}).
\]
This yields $|I|e^{-\sqrt{n}/2} + |I^c| \exp(-c\sqrt{n}) \geq 1$ and so this shows that we must have $|I|+|I^c| \geq e^{C \sqrt{n}}$ where $C > 0$ is an absolute positive constant. This completes the proof.

\section{Proof of Theorem \ref{thm:gw}}
\label{sec:proofgw}

\newcommand{\GG}{\mathbb{G}}
\newcommand{\N}{\mathcal{N}}

\newcommand{\la}{\left\langle}
\newcommand{\ra}{\right\rangle}
\def\beq{\begin{equation}}
\def\eeq{\end{equation}}
\newcommand{\laq}[1]{\label{eq:#1}} 

Before starting the proof, we need to recall some preliminary facts about the Gaussian width.

\paragraph{Preliminaries} Since the Gaussian width is translation invariant, it will be more convenient to work with $C = \Den - \frac{1}{n} I_n = \{X - \frac{1}{n} I_n : X \psd 0, \Tr(X) = 1\}$ which lives in the space $\S^n_0$ of $n\times n$ symmetric matrices of trace zero. We endow this space with the trace inner product $\langle A, B \rangle  = \Tr(AB)$ and we call $\N_0$ the standard Gaussian distribution associated to this inner product. A matrix $G \sim \N_0$ can also be defined as $G = \frac{1}{\sqrt{2}} (A - \frac{\Tr A}{n} I_n)$ where $A \sim GOE(n)$, the Gaussian Orthogonal Ensemble. Note that if $G \sim \N_0$ then the off-diagonal entries $G_{ij}$ for $i<j$ are independent and follow the distribution $\frac{1}{\sqrt{2}} N(0,1)$.

The Gaussian width of $C$ is easily computed by known results on the spectral norm of random matrices. Indeed we have: 
\[ w_G(C) = \E_{G \sim \N_0} \left[ \max_{x \in S^{n-1}} \la G , xx^T - \frac{1}{n} I_n \ra \right] = \E_{G \sim \N_0} [ \lambda_{\max}(G) ] \leq \sqrt{2n}. \]

We will also need the following result on the concentration of the support function of a convex set around its mean. It can be proved using standard Gaussian concentration results. The proof is given later.
\begin{lemma}
\label{lem:wGconcentration}
Assume $K \subset \RR^d$ is a convex set with Gaussian width $w_G(K)$. Then 
\begin{equation}
\label{eq:wgconcentration}
 \Pr_{g \sim N(0,I_d)} \left[ \max_{x \in K} \langle g,x \rangle \leq (1+\alpha) w_G(K) \right] \geq 1-\exp(-\alpha^2/(4\pi)).
 \end{equation}
\end{lemma}

We are now ready to prove Theorem \ref{thm:gw}. Assume $P$ is a polytope such that $C \subset P$ and $w_G(P) \leq 2w_G(C) \leq 2 \sqrt{2n} \leq 3\sqrt{n}$. We will show that the extension complexity of $P$ is at least $\exp(cn^{1/3})$ for some absolute constant $c > 0$.

\paragraph{Slack matrix} The upper bound on $w_G(P)$ allows us to generate valid inequalities for $P$. Indeed with large probability on $G \sim \N_0$ we know that $\max_{X \in P} \langle G , X \rangle \leq (1+\alpha) w_G(P)$. Also since we know that $C \subset P$ this allows us to construct a submatrix of the slack matrix of $P$. More precisely, define the following event:
\begin{equation}
\label{eq:GGevent}
\GG = \left\{G \in \S^n_0 :  -(1+\alpha)w_G(P) \leq \min_{X \in P} \langle G , X \rangle \text{ and } \max_{X \in P} \langle G , X \rangle \leq (1+\alpha)w_G(P)  \right\}.
\end{equation}
For large enough $\alpha > 0$, Lemma \ref{lem:wGconcentration} tells us that $\GG$ has positive probability when $G \sim \N_0$. If we fix $\alpha = 4$ for concreteness (for which $\Pr[\GG] > 0$), and recall that $w_G(P) \leq 3\sqrt{n}$ we get that for any $G \in \GG$, $\langle G , X \rangle + 15\sqrt{n} \geq \langle G , X \rangle + (1+\alpha) w_G(P) \geq 0$ is a valid inequality
 for $X \in P$.
Since by assumption $C \subset P$ then the (infinite) matrix
\begin{equation}
\label{eq:slackgw1}
(G,x) \in \GG \times S^{n-1} \mapsto \left\langle G , xx^T - \frac{1}{n} I_n \right\rangle + 15 \sqrt{n}
\end{equation}
is a submatrix of the slack matrix of $P$.
 By Yannakakis theorem, it follows that the extension complexity of $P$ is at least the nonnegative rank of the matrix \eqref{eq:slackgw1}. For the proof, it will be enough for us to consider $x \in \frac{1}{\sqrt{n}}\{-1,1\}^n$. Also note that since $\Tr(G) = 0$ for $G \in \GG$ then $\la G, xx^T - \frac{1}{n} \ra = x^T G x$.  The rest of the proof will thus be devoted to proving a lower bound on the nonnegative rank of the matrix:
\beq
\laq{slackgw}
S(G,x) = \frac{1}{15n\sqrt{n}} x^T G x + 1 \quad (G \in \GG, \; x \in H_n).
\eeq

\paragraph{Nonnegative rank of $S$} Consider a nonnegative factorization of $S$ of size $N$:
\beq
\laq{NMF-SGx}
S(G,x) = \sum_{i=1}^N g_i(G) f_i(x) \quad \forall G \in \GG, x \in H_n
\eeq
where $g_i,f_i \geq 0$. Since $\E_{x\in H_n}[S(G,x)] = 1$ for all $G \in \S^n_0$, we can normalize the $f_i$ and $g_i$ so that $\E[f_i] = 1$ and $\sum_{i=1}^N g_i \equiv 1$.

For any $G \in \S^n_0$ and $x \in H_n$, let $q_G(x) = x^T G x = 2 \sum_{i < j} G_{ij} x_i x_j$. The inner product of any two functions $f,g:H_n\rightarrow \RR$ is defined as $\langle f , g \rangle = \E_{x \in H_n} [ f(x) g(x) ]$. If we take the inner product of Equation \eqref{eq:NMF-SGx} with $q_G$ we get (using the definition of $S(G,x)$ in \eqref{eq:slackgw}):
\beq
\frac{1}{15 n\sqrt{n}} \|q_G\|_2^2 = \sum_{i=1}^N g_i(G) \langle f_i,q_G \rangle \quad \forall G \in \GG.
\eeq
We then take the expectation with respect to $G \sim \N_0 | \GG$ to get:
\beq
\laq{EGEx}
\frac{1}{15 n\sqrt{n}} \E_{G \sim \N_0 | \GG} \|q_G\|_2^2 = \E_{G\sim \N_0|\GG} \left[\sum_{i=1}^N g_i(G) \la f_i , q_G \ra \right].
\eeq
We now analyze the left-hand and right-hand sides of Equation \eqref{eq:EGEx}.

\begin{itemize}
\item \emph{Left-hand side of \eqref{eq:EGEx}}: For any fixed $G \in \S^n_0$ we have $\|q_G\|_2^2 = 4 \sum_{i < j} G_{ij}^2$. If $G \sim \N_0$ (without conditioning on $\GG$), then each $G_{ij} \sim \frac{1}{\sqrt{2}} N(0,1)$ and so we get $\E_{G\sim \N_0} \|q_G\|_2^2 = 4 \binom{n}{2} \frac{1}{2} = n(n-1)$. Now since $\GG$ is an event that has positive constant probability, and since $\|q_G\|_2^2$ concentrates around its mean (by standard concentration results for the squared Euclidean norm of Gaussian vectors), we know that we can assume $\E_{G \sim \N_0|\GG}[\|q_G\|_2^2] \geq c n(n-1)$ for some constant $0 < c < 1$, at the expense of replacing $\GG$ with the event $\GG \cap \{G : \|q_G\|_2^2 \geq c n(n-1)\}$ whose probability can still be bounded below by a positive constant. With this, the left-hand side of \eqref{eq:EGEx} is $\geq \frac{1}{15 n \sqrt{n}} c n(n-1) \geq c' \sqrt{n}$ where $c,c'$ are positive constants.
\item \emph{Right-hand side of \eqref{eq:EGEx}}: Just like in the proof of Theorem \ref{thm:main}, we will need to separately consider functions $f_i$ with large $\|\cdot \|_{\infty}$ norms from others. Let $M > 0$ be a parameter to be fixed later and let $I = \{i \in \{1,\ldots,N\} : \|f_i\|_{\infty} \geq M\}$. From the nonnegative factorization \eqref{eq:NMF-SGx} we have, for any $(G,x) \in \GG\times H_n$, $g_i(G) f_i(x) \leq S(G,x) \leq 2$ which gives $g_i(G) \leq 2/M$ for all $i \in I$ and $G \in \GG$. Thus for $i \in I$ and any $G \in \GG$ we have:
\[
g_i(G) \la f_i, q_G \ra \leq \frac{2}{M} \E[f_i] \max_{x \in H_n} q_G(x) \leq \frac{2}{M} 15 \sqrt{n}
\]
where in the last inequality we used the fact that $\E[f_i] = 1$ and $\max_{x \in H_n} q_G(x) \leq 15\sqrt{n}$ from the definition $G \in \GG$ in Equation \eqref{eq:GGevent}.
We can upper bound the RHS of \eqref{eq:EGEx} as follows:
\[
\begin{aligned}
\E_{G\sim \N_0|\GG} \left[\sum_{i=1}^N g_i(G) \la f_i , q_G \ra \right] &\leq 
\sum_{i \in I} \frac{2}{M} 15 \sqrt{n}
+
\E_{G\sim \N_0|\GG} \left[\sum_{i \in I^c} g_i(G) \la f_i , q_G \ra \right]\\
&\leq \frac{30}{M} \sqrt{n} N + \E_{G\sim \N_0|\GG} \left[ \max_{i \in I^c} \left|\la f_i , q_G \ra\right| \right]\\
&\leq \frac{30}{M} \sqrt{n} N + \frac{1}{\Pr[\GG]} \E_{G\sim \N_0} \left[ \max_{i \in I^c} \left|\la f_i , q_G \ra\right| \right],
\end{aligned}
\]
where in the second line we used the fact that $g_i(G) \geq 0$ and $\sum_{i \in I^c} g_i(G) \leq 1$.
For any function $f:H_n\rightarrow \RR$, the random variable $\la f_i, q_G \ra$, where $G \sim \N_0$, is a centered Gaussian and its variance is given by
\[
\E_G [ (\langle f , q_G \rangle)^2] = 4 \sum_{i<j} \E_G [G_{ij}^2] \langle f, \chi_{ij} \rangle^2 = \frac{1}{2} \sum_{i<j} \langle f, \chi_{ij} \rangle^2 = \frac{1}{2} \|\proj_2 f\|_2^2
\]
where we used the notation $\chi_{ij}(x) = x_i x_j$ and where $\proj_2 f$ is the homogeneous degree 2 component in the Fourier expansion of $f$. It is a well-known fact that if $X_1,\ldots,X_k$ are centered Gaussian random variables with variance at most $\sigma^2$ then $\E[\max_{i=1,\ldots,k} |X_i|] \leq \sqrt{2\log k}$. Using this we get that the RHS of \eqref{eq:EGEx} is at most:
\beq
\laq{ubrhs33}
\frac{30}{M} \sqrt{n} N + \frac{1}{\Pr[\GG]} \frac{1}{\sqrt{2}} \left(\max_{i \in I^c} \|\proj_2 f_i\|_2\right) \sqrt{2 \log N}.
\eeq
It remains to evaluate $\|\proj_2 f_i\|_2$ for $i \in I^c$. This is where we will use hypercontractivity and use the fact that $\|f_i\|_{\infty} \leq M$ for $i \in I^c$. The following lemma appears in \cite[Lemma 2.3]{regev2011quantum} and we include the proof below for convenience.
\begin{lemma}
\label{lem:h2bound}
Assume $f:H_n \rightarrow \RR$ satisfies $f \geq 0$, and $\E[f] = 1$. Assume furthermore that $\max f \leq M$. Then $\|\proj_{2} f\|_2 \leq e \log(M)$.
\end{lemma}
We can use this lemma to upper bound \eqref{eq:ubrhs33} by:
\beq
\laq{ubrhs44}
\frac{30}{M} \sqrt{n} N + \frac{1}{\Pr[\GG]} \frac{1}{\sqrt{2}} e \log(M) \sqrt{2 \log N}.
\eeq
\end{itemize}
To finish the proof we put together the lower and upper bounds above to get, recalling that $\Pr[\GG]$ is a positive constant:
\[
c \sqrt{n} \leq \frac{30}{M} \sqrt{n} N + c' \frac{1}{\sqrt{2}} e \log(M) \sqrt{2 \log N}.
\]
If we take $M = \frac{30}{c/2} N$ then the first term on the RHS is equal to $(c/2) \sqrt{n}$. This yields:
\[
(c/2) \sqrt{n} \leq c'' \log(N) \sqrt{\log(N)}
\]
which gives $N \geq \exp(cn^{1/3})$ as desired.

To finish it remains to prove Lemma \ref{lem:h2bound} and Lemma \ref{lem:wGconcentration}:

\begin{proof}[Proof of Lemma \ref{lem:h2bound}]
If $f$ has Fourier expansion $f = f_0 + f_1 + f_2 + \dots + f_n$ then
\[
\|\proj_2 f\|^2_2 = \frac{1}{\rho^4} (\rho^2 \|f_2\|_2)^2 \leq \frac{1}{\rho^4} \sum_{k} \rho^{2k} \|f_{k}\|_2^2 = \frac{1}{\rho^4} \|T_{\rho} f\|_2^2.
\]
Choose $\rho = \sqrt{p-1}$ with $1 \leq p \leq 2$ so that $\|T_{\rho} f\|_2 \leq \|f\|_{p}$ by hypercontractivity. Then we get:
\[
\|\proj_2 f\|_2 \leq \frac{1}{p-1} \|f\|_p.
\]
We have $\|f\|_p = (\E[f^p])^{1/p} \leq M^{(p-1)/p} \leq M^{p-1}$ where we used the fact that $f \geq 0$ and $\E[f] = 1$. This gives $\|\proj_2 f\|_2 \leq \frac{1}{p-1} M^{p-1}$. Choosing $p = 1 + \frac{1}{\log M}$ gives us $\|\proj_2 f\|_2 \leq e \log M$.
\end{proof}

\begin{proof}[Proof of Lemma \ref{lem:wGconcentration}]
We can assume without loss of generality that the origin is in the interior of $K$.
Let $h_K(u) = \max_{x \in K} \langle u , x \rangle = \|u\|_{K^{\circ}}$ be the support function of $K$. The function $h_K$ is $L$-Lipschitz with $L = \sqrt{2\pi} w_G(K)$. Indeed, by \cite[Proposition 7.5.2]{vershynin2018high}, $K \subseteq B(0, \sqrt{2\pi} w_G(K))$ where $B(0,R)$ is the Euclidean ball centered at the origin with radius $R$. It thus follows that $B(0,\frac{1}{L}) \subset K^{\circ}$. Thus for any $u,v$, $|h_K(u) - h_K(v)| = |\|u\|_{K^{\circ}} - \|v\|_{K^{\circ}}| \leq \|u-v\|_{K^{\circ}} \leq L \|u-v\|_2$ as needed.

Gaussian concentration tells us that if $f:\RR^d\rightarrow \RR$ is $L$-Lipschitz then $\Pr_{g\sim N(0,I_d)} [ f(g) \geq \E f + t ] \leq \exp\left(-\frac{t^2}{2L^2}\right)$, see e.g., \cite[Theorem 5.25]{ABMB}. We apply this result with $f = h_K$ and $t = \alpha w_G(K)$ which gives $\Pr[h_K(g) \geq (1+\alpha) w_G(K)] \leq \exp(-\frac{-\alpha^2}{4\pi})$ as desired.
\end{proof}

\paragraph{Acknowledgments} I would like to thank Omar Fawzi for his encouragements and for bringing the paper \cite{regev2011quantum} to my attention. I would also like to thank Pablo Parrilo who suggested looking at approximation in terms of the Gaussian width, which led to Theorem \ref{thm:gw}. I am grateful to Guillaume Aubrun for discussions on an earlier draft of the paper, and for showing me how Theorem \ref{thm:main} applies with general $A$ and giving me the permission to include his argument (previously Theorem \ref{thm:main} only dealt with the case $A=\frac{1}{n} I_n$). I would also like to thank Samuel Fiorini for discussions on \cite{braun2012approximation} and James Saunderson and Pravesh Kothari for various discussions related to this manuscript.

\bibliography{../../bib/nonnegative_rank}
\bibliographystyle{alpha}

\end{document}